\documentclass[aps,prd,showkeys,superscriptaddress,nofootinbib,floatfix,singlecolumn]{revtex4-2}

\usepackage{amsmath,amsthm,amssymb}
\usepackage{mathtools}
\usepackage{mathrsfs}
\newcommand{\be}{\begin{equation}}
\newcommand{\ee}{\end{equation}}
\newcommand{\bea}{\begin{eqnarray}}
\newcommand{\eea}{\end{eqnarray}}

\newcommand{\R}{\mathbb{R}}
\newtheorem{theorem}{Theorem}[section]

\usepackage[dvipsnames, usenames]{xcolor}

\begin{document}

\title{Local well-posedness and singularity formation in non-Newtonian compressible fluids}

\author{Ariel Lerman}
\email{ari.lerman@yale.edu}
\affiliation{Department of Mathematics, University of Illinois at Urbana-Champaign, Urbana, IL 61801, USA}
\affiliation{Department of Mathematics, Yale University, New Haven, Connecticut 06511, USA}

%senior authors listed alphabetically

\author{Marcelo M.\ Disconzi}
\email{marcelo.disconzi@vanderbilt.edu}
\affiliation{Department of Mathematics, Vanderbilt University, Nashville, TN, 37211, USA}

\author{Jorge Noronha}
\email{jn0508@illinois.edu}
\affiliation{Illinois Center for Advanced Studies of the Universe\\ Department of Physics, 
University of Illinois at Urbana-Champaign, Urbana, IL 61801, USA}

\begin{abstract}
We investigate the initial value problem of a very general class of $3+1$ non-Newtonian compressible fluids in which the viscous stress tensor with shear and bulk viscosity relaxes to its Navier-Stokes values. These fluids correspond to the non-relativistic limit of well-known Israel-Stewart-like theories used in the relativistic fluid dynamic simulations of high-energy nuclear and astrophysical systems. After establishing the local well-posedness of the Cauchy problem, we show for the first time in the literature that there exists a large class of initial data for which the corresponding evolution breaks down in finite time due to the formation of singularities. This implies that a large class of non-Newtonian fluids do not have finite solutions defined at all times. 
\end{abstract}

%\today

\maketitle

\section{Introduction}

Fluid dynamics \cite{LandauLifshitzFluids} plays a crucial role in the understanding of a variety of physical phenomena defined at a multitude of length scales, ranging from the large-scale structure properties of the universe \cite{WeinbergCosmology} to subatomic quark-gluon plasmas formed in colliders \cite{Romatschke:2017ejr}. The ubiquitousness of fluid dynamics stems from the fact that the equations of motion that govern the hydrodynamic variables can be derived from very simple, yet fundamental, conservation laws (mass, energy, and momentum) \cite{LandauLifshitzFluids} together with general assumptions regarding the locality of interactions \cite{doi:10.1126/science.284.5411.87}. Most fluids in nature can be described this way through the celebrated Navier-Stokes equations \cite{LandauLifshitzFluids}, which describe how the fluid mass density, velocity, and temperature evolve in space and time in the presence of external and internal viscous/dissipative forces. Furthermore, understanding the mathematical properties of the Navier-Stokes equations is also a fundamental problem in modern mathematics.

Many fundamental questions also remain in fluid dynamics in the relativistic regime. The basic equations of relativistic fluid dynamics stem from the conservation of energy and momentum, $\partial_\mu T^{\mu\nu}=0$, where $T^{\mu\nu}$ is the system's energy-momentum tensor, and the dynamics of a conserved current, $\partial_\mu J^\mu=0$, which describes, for example, baryon number conservation. In the absence of dissipation (i.e., in the zero entropy production limit), this defines a set of five nonlinear partial differential equations (PDEs) that describe the time and spatial evolution of the system's energy density, baryon density, and flow velocity \cite{Rezzolla_Zanotti_book}. Several properties \cite{ChoquetBruhatGRBook} are known about the so-called relativistic perfect fluid equations, and a detailed account of all the developments in this field is beyond the scope of this work (see, for example, \cite{Disconzi:2023rtt} for a recent review). However, a few results are worth mentioning here as they are connected to, or have directly influenced, the work presented in this paper. For example, conditions for the local well-posedness of the Cauchy problem (assuming standard physical constraints on the equation of state) can be found in \cite{ChoquetBruhatGRBook} by rewriting the system of equations as a first-order symmetric hyperbolic (FOSH) set of partial differential equations (PDEs). In this regard, we emphasize the vital role of geometric-analytic techniques from the theory of quasilinear wave equations in studying relativistic perfect fluids. This is also illustrated by the work presented in Ref.\ \cite{2019AnHP...20.2173D} where a new formulation of the relativistic perfect fluid equations tailored to the
characteristics of the system (the sound cones and the flow lines) has been used to investigate complex mathematical questions for which detailed information
about the behavior of solutions is needed. This led to new insights into the problem of shock formation for relativistic perfect fluids \cite{Abbrescia-Speck-2023-arxiv} and improved regularity of
solutions \cite{Sifan-Yu-2022-arxiv}. Further developments concerning the local well-posedness theory for the free boundary relativistic Euler equations with a physical vacuum boundary on a Minkowski background can be found in \cite{2022ArRMA.245..127D}. 
Global well-posedness for the relativistic Euler equations in an expanding background has been proven in 
\cite{MR3101792} (see \cite{SpeckNonlinearStability, MR3335528, MR3120746} for the case with coupling to Einstein's equations and also \cite{Oliynyk-2021,Fajman-Oliynyk-Wyatt-2021}).
Despite such developments, important questions remain open in the theory of relativistic ideal fluids, including a detailed description of the long-time behavior of shock-forming solutions 
to the relativistic Euler equations \cite{ChristodoulouShocks,Christodoulou-Book-2019}, local well-posedness of the free-boundary Einstein-Euler system \cite{RendallFluidBodies,Disconzi:2023rtt}, and a description of anomalous dissipation in relativistic ideal fluids and its connection to relativistic turbulence \cite{Eyink:2017zfz}, to mention just a few examples.

The search for a consistent formulation of dissipative effects in relativistic fluids still attracts significant interest from physicists \cite{Romatschke:2017vte} and mathematicians \cite{Disconzi:2023rtt}. This is due, in part, to the fact that the standard generalization of the Navier-Stokes equations to the relativistic domain, pioneered by Eckart \cite{EckartViscous} and Landau and Lifshitz \cite{LandauLifshitzFluids}, display unwanted features. These theories have unphysical behavior manifested via causality violation and the fact that in such formulations, the global equilibrium state is generally unstable to linear perturbations \cite{Hiscock_Lindblom_instability_1985}. The latter is now understood to follow from the very general (and intuitive result) only proved recently in \cite{Bemfica:2020zjp,Gavassino:2021owo}, which states that in a relativistic system, causality is a necessary condition for the stability of the equilibrium state. The formulations from Eckart, Landau, and Lifshitz give rise to non-hyperbolic PDEs \cite{PichonViscous}, which necessarily violate relativistic causality, thus allowing for the appearance of unphysical instabilities of the equilibrium state.

Only very recently a general formulation of these so-called first-order theories has been proposed by Bemfica, Disconzi, Noronha, and Kovtun (BDNK) \cite{BemficaDisconziNoronha,Kovtun:2019hdm,Bemfica:2019knx,Bemfica:2020zjp}, in which conditions are given that ensure that the dynamics is causal and strongly hyperbolic in the fully nonlinear regime (including shear, bulk, and conductivity effects), and the equilibrium state is stable against small perturbations. The BDNK formulation consistently generalizes previous developments presented in \cite{EckartViscous,LandauLifshitzFluids,Van:2007pw,Tsumura:2007wu,TempleViscous,TempleViscous2,TempleViscous3} where it was understood that the issues found in the theories of Eckart and Landau and Lifshitz were connected to problematic definitions of the hydrodynamic fields in an out of equilibrium state. In BDNK, following the standard effective theory reasoning \cite{Kovtun:2012rj}, non-equilibrium corrections are described in terms of the most general expressions compatible with the symmetries involving first-order spacetime derivatives of the standard hydrodynamic variables (e.g., temperature, fluid velocity, and chemical potential), which vanish in equilibrium. This generates several new terms in the equations of motion, which were not considered in the Eckart and Landau-Lifshitz formulations. These new terms are crucial to recovering relativistic causality and stability. As a matter of fact, for BDNK theories, local well-posedness has also been established for the system of PDEs describing the viscous fluid coupled to Einstein's equations in \cite{BemficaDisconziNoronha,Bemfica:2019knx,Bemfica:2020zjp,Bemfica-Disconzi-Rodriguez-Shao-2021,Bemfica-Disconzi-Graber-2021}. Shockwave solutions {\cite{Freistuhler:2021lla,Pellhammer-2023-arxiv}, as well as general numerical solutions to such fluid dynamic theories are currently being developed, and progress in this direction can be found in \cite{Pandya:2021ief,Pandya:2022pif,Pandya:2022sff,Bantilan:2022ech,Rocha:2022ind}. Even though local well-posedness has been established, very little is known about the global aspects of solutions of this theory. While small-data global well-posedness for the BDNK system has been recently proven in \cite{Sroczinski-2022-arxiv} for the case of a conformal fluid in flat space-time (which is likely the simplest possible setting), it remains a challenge to generalize to 
the BDNK system most of the breakthroughs obtained for relativistic perfect fluids mentioned above. Therefore, when it comes to studying the global well-posedness of solutions, it is conceivable that formulations of relativistic viscous fluid dynamics naturally\footnote{Of course, one can define new variables and reduce the order of a 2nd-order set of PDEs to find a new set of 1st-order differential equations. However, this process is not guaranteed to generate a first-order system with good structure.} described in terms of a set of 1st-order transport-like PDEs, more akin to the relativistic Euler equations, may provide a  better route. Luckily, such formulations already exist, as we discuss below.

Relativistic viscous fluid dynamical modeling is the primary tool to describe the evolution of the quark-gluon plasma created in ultra-relativistic heavy-ion collisions \cite{Heinz:2013th}. Current simulations (for a review, see \cite{Romatschke:2017ejr}) are based on equations of motion that incorporate features and ideas from the seminal work by Israel and Stewart \cite{MIS-6}. These approaches, see \cite{Baier:2007ix,Denicol:2012cn}, differ from the BDNK reasoning mentioned above because the dissipative terms are not expressed solely in terms of the hydrodynamic variables and their derivatives. Instead, in those theories\footnote{Generalizations of this approach, aimed at describing systems even farther from equilibrium, have also been developed in the context of relativistic anisotropic hydrodynamics \cite{Alqahtani:2017mhy}.}, the dissipative fluxes obey additional equations of motion (derived either from truncations in kinetic theory \cite{Denicol:2012cn}, effective theory arguments \cite{Baier:2007ix}, or extended thermodynamics \cite{MuellerRuggeriBook,JouetallBook}) that describe how dissipative quantities evolve towards their first-order, asymptotic behavior described by relativistic Navier-Stokes theory. This process is characterized by a relaxation time that effectively implements causality by introducing a time delay for the dissipative currents to approach Navier-Stokes values -- causality imposes that flow gradients cannot be automatically converted into acceleration (as it occurs in Navier-Stokes equations). Therefore, in practice, causality requires relativistic fluids to be non-Newtonian fluids. This is the case not only for Israel-Stewart systems since BDNK fluids also display the same property \cite{Bemfica:2020zjp}. 

Despite the existence of numerical simulations developed by the heavy-ion physics community, the general properties displayed by solutions of the Israel-Stewart relativistic hydrodynamics equations are still broadly not understood.  Nonlinear conditions for causality involving shear and bulk viscosity have been derived in \cite{Bemfica:2020xym} and investigated in numerical simulations of heavy-ion collisions in \cite{Chiu:2021muk,Plumberg:2021bme}. The conditions under which Israel-Stewart-like theories involving shear, bulk, and heat conductivity/particle diffusion effects are causal and strongly hyperbolic in the fully nonlinear regime remain unknown. The only exception is when only bulk viscosity is present. In that case, strong hyperbolicity and local well-posedness have been established in the nonlinear regime in Ref.\ \cite{Bemfica:2019cop}. This is possible because, in the case of pure bulk viscosity, the equations of motion are considerably simpler than the full case where shear and particle diffusion are included. In fact, bulk-viscous effects in Israel-Stewart theories are described by a single Lorentz scalar, and its inclusion in the dynamics amounts to adding a single first-order relaxation-type equation to be solved together with the conservation laws. The system of equations corresponds to a set of nonlinear 1st-order PDEs that is not so different than the relativistic perfect fluid. Thus, it is conceivable that global well-posedness analyses of solutions of the equations of motion can also be made.

Progress in this direction was obtained in Ref.\ \cite{2020arXiv200803841D} where it was shown that in Israel-Stewart-like theories with bulk viscosity, there exists a class of smooth initial data for which the corresponding solutions to the Cauchy problem break down in finite time. Ref.\ \cite{2020arXiv200803841D} showed that, for appropriate data, there exists a finite $\mathcal{T} > 0$ such that, for times $t>\mathcal{T}$ solutions either cannot remain $C^1$ or they become acausal. Further work is needed to precisely determine the nature of this unphysical behavior (for example, whether this implies shock formation). In any case, it is important to stress that no results of this kind exist when the other dissipative effects are considered, such as shear viscosity or heat conduction. 

This work sheds new light on this problem from a different perspective. As mentioned above, in relativity, the dissipative quantities (shear-stress tensor, bulk viscous scalar, heat flow) cannot immediately take their Navier-Stokes values because of causality (for example, the conversion of flow gradients into acceleration is bounded by the speed of light). This implies that every causal and stable relativistic fluid is necessarily a non-Newtonian fluid. In this paper, we use this fact to take the first steps towards understanding the global behavior of nonlinear solutions of relativistic viscous fluids with bulk \emph{and} shear viscosity by investigating the same problem in an \emph{analog} problem corresponding to a similar non-Newtonian fluid defined by Israel-Stewart-like theories in the \emph{non-relativistic regime}\footnote{For simplicity, we only consider compressible systems with barotropic equations of state (so that one does not need to consider the dynamics of the temperature).}. We determine the conditions under which strong hyperbolicity and local well-posedness hold for the set of equations of motion with shear and bulk viscosity in three spatial dimensions considered in this paper. For typical choices of transport coefficients, we provide a proof \emph{\`a la} Sideris \cite{1985CMaPh.101..475S} that $C^{1}$ solutions in these viscous fluids lose regularity at a finite time. This is the first time such a general result has been proven for a large class of non-Newtonian fluids. Since they correspond to the non-relativistic regime of the Israel-Stewart theories currently used in the description of the quark-gluon plasma formed in heavy-ion collisions, our work shows, for the first time, that the regions in the quark-gluon plasma where the flow velocity approaches the non-relativistic limit can lose regularity at finite time.

Furthermore, the same can be said about the slow (i.e., non-relativistic) moving parts of the matter formed in neutron star mergers that have been described in terms of Israel-Stewart-like theories \cite{Camelio:2022ljs,Camelio:2022fds,Most:2022yhe,Chabanov:2023blf}. Therefore, our work not only unveils new mathematical properties of a large class of non-Newtonian fluids but also leads to a new fundamental understanding of the non-relativistic regime of relevant systems in high-energy nuclear physics and astrophysics. Furthermore, we emphasize that the non-Newtonian fluids we consider are also interesting in their own right as they appear in extended thermodynamic models \cite{JouetallBook} and also in the post-Newtonian expansion of relativistic fluids in the inspiral phase of neutron star mergers \cite{HegadeKR:2023glb}.

This paper is organized as follows. In Section \ref{section2}, we discuss the relativistic Israel-Stewart-like theory with bulk viscosity and define the equations of motion corresponding to its non-relativistic regime. We show that linearized disturbances around global equilibrium are stable in Section \ref{bulklocal}. Conditions are found in \ref{section4} for the equations of motion to admit a first-order symmetric hyperbolic formulation, which implies local well-posedness of the initial value problem with initial data suitably defined in a Sobolev space $H^N$ for sufficiently large $N$ \cite{KatoQuasiLinear}. Section \ref{GWPSec} shows how arguments originally derived by Sideris \cite{1985CMaPh.101..475S} for an ideal fluid could be used to prove that solutions of the non-relativistic viscous equations with bulk viscosity develop singularities at a finite time. We generalize these results to include shear viscosity in Section \ref{GWPshear}. After presenting the relevant computations, our treatment of the shear case closely follows that of the case with only bulk
and, thus, will be presented in a very concise manner.

\emph{Notation:} We use natural units $c = \hbar = k_B = 1$ and a mostly plus Minkowski metric $g_{\mu\nu}$. Greek indices run from 0 to 3, while Latin indices run from 1 to 3. When convenient, three-dimensional vectors $v^i$ are also denoted as $\bf{v}$, and inner products $v^i a_i = \bf{v}\cdot \bf{a}$.

%%%%%%%%%%%%%%%%%%%%%%%%%%%%%%%%%%%%%%%%%%%%%
\section{Non-relativistic Israel-Stewart-like equations with bulk viscosity}\label{section2}
%%%%%%%%%%%%%%%%%%%%%%%%%%%%%%%%%%%%%%%%%%%%%

This section briefly explains how one can derive the non-relativistic regime of the Israel-Stewart-like theory for bulk viscosity considered in \cite{Bemfica:2019cop}. This model describes bulk-viscous effects in a variety of systems, such as the quark-gluon plasma formed in heavy-ion collisions \cite{Romatschke:2017vte} and the chemical equilibration processes associated with flavor-changing reactions in the dense matter formed in neutron star mergers \cite{Camelio:2022ljs,Most:2022yhe,Gavassino:2023xkt}.

The relativistic fluid is defined by an energy-momentum tensor
\be
T_{\mu\nu} = (\varepsilon+P + \Pi)u_\mu u_\nu +(P+\Pi) g_{\mu\nu} 
\ee
and a conserved rest-mass current 
\be
J^\mu = \rho u^\mu.
\ee
Above, $\varepsilon$ is the total energy density, $P$ is the equilibrium pressure defined by the equation of state, $u^\mu$ is the 4-velocity of the fluid (which obeys $u_\mu u^\mu = -1$), $\rho$ is the rest-mass density, and $\Pi$ is the bulk scalar ($\Pi = 0$ in equilibrium). The evolution of the fluid is defined by the equations of motion that stem from the conservation of energy and momentum
\begin{equation}
\partial_\mu T^{\mu\nu} = 0,
\label{conservationEM}
\end{equation}
the conservation of the rest-mass current
\be
\partial_\mu J^\mu = 0,
\label{conservationbaryon}
\ee
and the additional relaxation-type equation for the bulk scalar
\begin{equation}
    \tau_\Pi u^\alpha \partial_\alpha \Pi + \delta_{\Pi\Pi}\Pi \,\partial_\alpha u^\alpha + \lambda \Pi^2 + \zeta \partial_\alpha u^\alpha = 0,
    \label{ISeq1}
\end{equation}
where $\zeta$ is the bulk viscosity coefficient, $\tau_\Pi $ is the bulk relaxation time, and $\delta_{\Pi\Pi}$ and $\lambda$ are second-order transport coefficients \cite{Denicol:2012cn}. It was shown in \cite{Bemfica:2019cop} that this relativistic fluid is causal, strongly hyperbolic, and the initial-value problem (IVP) is locally well-posed (LWP) if 
\be
\left[\frac{\zeta}{\tau_\Pi} + \frac{\delta_{\Pi\Pi}}{\tau_\Pi}\Pi + \rho \left(\frac{\partial P}{\partial \rho}\right)_\varepsilon\right]\frac{1}{\varepsilon+P+\Pi} \leq 1 - \left(\frac{\partial P}{\partial \varepsilon}\right)_\rho.
\label{eqPRLbulk}
\ee
We note that, above, the coefficients $\tau_\pi,\delta_{\Pi\Pi},\zeta, \lambda$ can depend on $\varepsilon$, $\rho$, and also $\Pi$. Therefore, \eqref{eqPRLbulk} actually defines a vast class of systems determined by how the transport coefficients depend on the dynamical variables. Causality is determined from the system's characteristics \cite{Bemfica:2019cop}. Strong hyperbolicity \cite{ReulaStrongHyperblic} follows from the result demonstrated in \cite{Bemfica:2019cop} that the full nonlinear set of equations of motion (even when dynamically coupled to Einstein's equations) can be written in FOSH form. The fact that this implies that the IVP is LWP follows from well-known arguments that can be found in standard references, such as \cite{Courant_and_Hilbert_book_2}.

This work investigates what can be said about this system in the non-relativistic regime. This limit of Eqs.\ \eqref{conservationEM}, \eqref{conservationbaryon}, and \eqref{ISeq1} is obtained as follows. First, recall that in the non-relativistic limit\footnote{Here, we will not investigate the non-relativistic limit from a rigorous perspective since our goal with such limit is to point out a formal connection between IS-like theories and the non-Newtonian fluids we consider. Our results below regarding the evolution problem for the non-Newtonian fluid, on the other hand, are mathematically rigorous.}, one uses the 3-velocity $v_i$ instead of the 4-velocity $u^\mu = \gamma_v(1,v^i)$, where $\gamma_v = 1/\sqrt{1-v_i v^i}$. In fact, one may simply drop the $\gamma_v$ and use in this regime $u^\mu \to (1,v^i)$ such that $u^\mu \partial_\mu \to \partial_t + v_i \partial^i$ (the material derivative). The same type of approximation implies $\partial_\mu u^\mu \to \partial_i v^i$. Thus, in this limit, conservation of the rest-mass current \eqref{conservationbaryon} becomes
\be
\partial_t \rho + \partial_i (\rho v^i)=0.
\label{nonrel_eq_mot1}
\ee
Using that in a non-relativistic regime the energy density of the fluid is essentially given by the rest-mass density and that the total pressure contribution to the energy density is negligible \cite{Rezzolla_Zanotti_book}, we find that $u_\nu \partial_\mu T^{\mu\nu}=0$ leads to Euler's equation for the velocity including a bulk viscous term
\be
\rho \partial_t v_i + \rho v^k \partial_k v_i + c_s^2\partial_i \rho + \partial_i \Pi = 0,
\label{nonrel_eq_mot2}
\ee
where we assumed that $P$ only depends on $\rho$, and defined the speed of sound 
\be
c_s = \sqrt{\frac{\partial P}{\partial \rho}},
\end{equation}
assumed to be nonvanishing and finite.
In this case, the energy conservation equation does not add any new information \cite{LandauLifshitzFluids}. Finally, assuming that the transport coefficients survive the Newtonian limit\footnote{In principle, one may devise other schemes where $\tau_\Pi$ vanishes in the non-relativistic regime, see \cite{HegadeKR:2023glb}. We will not consider this case here.}, one obtains for the Israel-Stewart equation for $\Pi$ \eqref{ISeq1} the following non-relativistic equation 
\be
\tau_\Pi \partial_t \Pi + \tau_\Pi \partial_i (v^i \Pi) + \Pi + \zeta \partial_i v^i=0,
\label{nonrel_eq_mot3}
\ee
where we assumed for simplicity that $\delta_{\Pi\Pi}=\tau_\Pi$ and $\lambda=0$. 
%\bea 
%&&\frac{1}{\rho}\partial_t \rho +  \nabla_k v_k + \frac{v_k}{\rho} \nabla_k \rho = 0 \label{nonrel_eq_mot1}\\ 
%&&\frac{\rho}{c_s^2} \partial_t v_i +  \nabla_i \rho+ \frac{\rho}{c_s^2} \, v_k \nabla_k v_i  + \frac{1}{c_s^2}\nabla_i \Pi =0 \label{nonrel_eq_mot2}\\
%&&\frac{\tau_\Pi}{\zeta c_s^2}\partial_t \Pi +  \frac{1}{c_s^2}  \nabla_k v_k+ \frac{\tau_\Pi}{\zeta c_s^2} v_k \nabla_k \Pi +\frac{\Pi}{\zeta c_s^2} =0.\label{nonrel_eq_mot3}
%\eea
Equations \eqref{nonrel_eq_mot1}, \eqref{nonrel_eq_mot2}, and \eqref{nonrel_eq_mot3} define the non-relativistic limit of the Israel-Stewart theory with bulk viscosity considered in this paper. These constitute a system of nonlinear PDEs for the variables $\rho$, $v^i$, and $\Pi$. They describe the motion of a class of non-Newtonian fluids where the bulk scalar obeys a relaxation equation, which can be very general since $\zeta$ and $\tau_\Pi$ may depend on $\Pi$. We note that the Navier-Stokes equations for compressible fluid with bulk viscosity (but no shear viscosity) \cite{LandauLifshitzFluids} are obtained from \eqref{nonrel_eq_mot3} in the limit where $\tau_\Pi \to 0$ and $\zeta$ is only a function of $\rho$.

%%%%%%%%%%%%%%%%%%%%%%%%%%%%%%%%%%%%%%%%%%%%%
\section{Local Stability around the equilibrium state}\label{bulklocal}
%%%%%%%%%%%%%%%%%%%%%%%%%%%%%%%%%%%%%%%%%%%%%

We now investigate the local stability of equations \eqref{nonrel_eq_mot1}, \eqref{nonrel_eq_mot2}, and \eqref{nonrel_eq_mot3} against small disturbances around the equilibrium state. This analysis is standard, and we included it here for completeness. The goal is to show that the equilibrium state in the model is stable for all Galilean observers. This mirrors typical stability analyses done in the relativistic regime \cite{Hiscock_Lindblom_instability_1985}, where stability refers to the equilibrium state observed in any Lorentz frame. 

Assume that the system starts in a global equilibrium state where $\Pi = 0$ and $\rho = \rho_0>0$ and $v_i = v_i^0$ are constants, and is disturbed by a small deviation $\delta \rho$, $\delta v_i$, $\delta \Pi$. A system is locally stable if it returns to equilibrium, assuming the deviation is sufficiently small. This is quantifiable by looking at the Fourier transform of the equations: if the imaginary part of the frequency of the modes is negative, then the deviations will decrease in time and hence return the system to equilibrium.

Assuming these deviations are sufficiently small, we may eliminate terms that are not first-order, as the higher-order terms are insignificant for sufficiently small deviations. This leads  to the following equations:
\begin{align}
    \partial_t \delta \rho + \rho_0 \partial_k \delta v^k + v_k^0 \partial^k \delta \rho &= 0, \label{eq16}\\
    \rho_0\partial_t \delta v_i + c_s^2\partial_i \delta \rho + \rho_0 v_k^0 \partial^k \delta v_i + \partial_i \delta \Pi &= 0, \label{eq17}\\
   \tau_\Pi\partial_t \delta \Pi + \zeta \partial_k \delta v^k + \tau_\Pi v_k^0 \partial^k \delta \Pi + \delta \Pi &= 0. \label{eq18}
\end{align}
We note that $\zeta$, $c_s$, and $\tau_\Pi$ are constant in the equations above. This comes from the fact that we are linearizing the system of equations around the uniform equilibrium state. 

We look for solutions of the kind $\sim \exp\left(-i\omega t + i\,\mathbf{k}\cdot \mathbf{x}\right)$ by taking the Fourier transform of the equations of motion in both space and time, using variables $\omega$ for frequency and $\mathbf{k}$ for wavenumber:
\bea
    - ( \omega - \mathbf{v}_0 \cdot \mathbf{k}) \delta \rho +  \rho_0\, \mathbf{k} \cdot  \delta \mathbf{v}  &=& 0, \label{first}\\
     -  \rho_0 (\omega - \mathbf{v}_0 \cdot \mathbf{k})\delta v_j +  c_s^2 k_j  \delta \rho  + k_j \delta \Pi &=& 0, \label{sec}\\
    -i(\omega - \mathbf{v}_0 \cdot \mathbf{k}) \tau_\Pi \delta \Pi +   \delta \Pi + i\zeta\, \mathbf{k} \cdot \delta \mathbf{v} &=& 0, \label{third}
\eea
where above, to ease the notation, we have not distinguished $\delta \rho(t,\mathbf{x})$ from its Fourier transform $\delta \rho(\omega,\mathbf{k})$.

To better analyze the system's stability, let $\Omega$ be defined by $\Omega = \omega - \mathbf{v}_0 \cdot \mathbf{k}$. We note that, as $\mathbf{v}_0$ and $\mathbf{k}$ are real-valued, and the imaginary parts of $\Omega$ and $\omega$ are equal.
With this new notation, Eqs.\ \eqref{first} and \eqref{third} give either trivial solutions where $\mathbf{k}$ is orthogonal to $\delta \mathbf{v}$ and $\Omega = 0$, a degenerate solution, 
or allow us to solve for $\delta \rho$ and $\delta \Pi$ in terms of $\delta \mathbf{v}$. We use these solutions to combine our equations, obtaining
\begin{equation}
-i \tau_\Pi \Omega^3  + \Omega^2 - \mathbf{k}^2 c_s^2 + i \,\tau_\Pi\, \mathbf{k}^2 \Omega \left(c_s^2 + \frac{\zeta}{\rho_0}\right) = 0.
\end{equation}
This equation determines the dispersion relations of the modes, $\omega=\omega(\mathbf{k})$.  
To obtain local stability, we need the imaginary part of $\Omega$ to be negative or the real part of $-i \Omega $ (that is, $-i \omega$) to be negative. This condition implies exponential decay of the variables $\delta \rho$, $\delta v_i$, $\delta \Pi$ after taking the inverse Fourier transform. Since this equation can be viewed as a cubic in $-i \Omega$, we can  use the Routh-Hurwitz criterion \cite{gradshteyn2007} after bringing the equation into the desired form, with a positive constant term:
\begin{equation}
    \tau_\Pi (-i \Omega)^3 + (-i \Omega)^2 + \tau_\Pi \mathbf{k}^2 \left(c_s^2 + \frac{\zeta}{\rho_0}\right) (-i \Omega) + \mathbf{k}^2 c_s^2 = 0.
\end{equation}
Thus, applying the Routh-Hurwitz criterion, we see that the system is stable if $\Delta_i > 0$ for $i = 1, 2, 3$, where
\begin{align}
    \Delta_1 &= \mathbf{k}^2 c_s^2\\
    \Delta_2 &= \det \begin{pmatrix}
    \tau_\Pi \mathbf{k}^2 \left(c_s^2 + \frac{\zeta}{\rho_0}\right) & \mathbf{k}^2 c_s^2\\
    \tau_\Pi & 1
    \end{pmatrix}\\
    \Delta_3 &= \tau_\Pi \Delta_2.
\end{align}
Since $\mathbf{k}$ and $c_s$ are real, this condition is equivalent to
\begin{align}
    \tau_\Pi &> 0\\
    \tau_\Pi \frac{\zeta}{\rho_0} &> 0.
\end{align}
As $\tau_\Pi$ is the relaxation time, the condition above agrees with our intuition. Similarly, the density is assumed to be positive everywhere, and thus, we derive that stability implies that $\zeta > 0$ also. Therefore, disturbances near the equilibrium state are locally stable when $\tau_\Pi, \zeta, \rho > 0$ (note that this is valid also when $\tau_\Pi$ and $\zeta$ are nontrivial functions of the density). Hence, we have shown the local stability of these equations around the equilibrium state under basic physical assumptions. While further studies could be made about the linearized equations, this is beyond the scope of this paper as our goal here in this section is to establish that our equations display the expected behavior near equilibrium, where fluid dynamics is most easily understood.

\section{Hyperbolicity and Local well-posedness}\label{section4}

In this section, we show that the set of equations of motion \eqref{nonrel_eq_mot1}, \eqref{nonrel_eq_mot2}, and \eqref{nonrel_eq_mot3} for our non-relativistic Israel-Stewart model with
bulk viscosity is strongly hyperbolic and, hence, locally well-posed given suitable initial data. The proof below is quite general, as we do not assume that $\zeta$ and the relaxation time $\tau_\Pi$ depend only on the mass density $\rho$ (i.e., they may also depend on $\Pi$). This implies that our proof is valid for a very general class of non-Newtonian fluids parameterized by the dependence of the transport coefficients on the bulk stress. 

We rewrite this system of equations \eqref{nonrel_eq_mot1}, \eqref{nonrel_eq_mot2}, and \eqref{nonrel_eq_mot3} in quasilinear form as follows 
\be
\mathcal{A}^0 \partial_t \Phi + \mathcal{A}^1 \partial_1 \Phi + \mathcal{A}^2 \partial_2 \Phi+\mathcal{A}^3 \partial_3 \Phi + \mathcal{B} \Phi = 0,
\label{definePDE}
\ee
where $\Phi = \{\rho,v_1,v_2,v_3,\Pi\}$ and, in this case, $\mathcal{A}^0$, $\mathcal{A}^1$, $\mathcal{A}^2$, $\mathcal{A}^3$, $\mathcal{B}$ are $5 \times 5$ symmetric real-valued matrices given by
\bea
&&\mathcal{A}^0 = \begin{pmatrix}
 \frac{1}{\rho}  & 0 & 0 & 0 & 0  \\
  0 & \frac{\rho}{c_s^2} & 0 & 0 & 0\\
  0 & 0 & \frac{\rho}{c_s^2} & 0 & 0\\
  0 & 0 & 0 & \frac{\rho}{c_s^2} & 0 \\
   0 & 0 & 0 & 0 & \frac{\tau_\Pi}{\zeta c_s^2}
   \end{pmatrix},
\mathcal{A}^1 = \begin{pmatrix}
 \frac{v_1}{\rho}  & 1 & 0 & 0 & 0  \\
  1 & v_1\frac{\rho}{c_s^2} & 0 & 0 & \frac{1}{c_s^2}\\
  0 & 0 & v_1\frac{\rho}{c_s^2} & 0 & 0\\
  0 & 0 & 0 & v_1\frac{\rho}{c_s^2} & 0\\
  0 & \frac{1}{c_s^2} & 0 & 0 & \frac{\tau_\Pi v_1}{\zeta c_s^2}\\
   \end{pmatrix},
\mathcal{A}^2=\begin{pmatrix}
 \frac{v_2}{\rho}  & 0 & 1 & 0 &0  \\
  0 & v_2\frac{\rho}{c_s^2} & 0 & 0 & 0\\
  1 & 0 & v_2\frac{\rho}{c_s^2} & 0 & \frac{1}{c_s^2}\\
  0 & 0 & 0 & v_2\frac{\rho}{c_s^2} & 0 \\
  0 & 0 & \frac{1}{c_s^2} & 0 & \frac{\tau_\Pi v_2}{\zeta c_s^2} 
   \end{pmatrix}
\nonumber \\ &&\mathcal{A}^3= \begin{pmatrix}
 \frac{v_3}{\rho}  & 0 & 0 & 1 &0  \\
  0 & v_3\frac{\rho}{c_s^2} & 0 & 0 & 0\\
  0 & 0 & v_3\frac{\rho}{c_s^2} & 0 & 0\\
  1 & 0 & 0 & v_3\frac{\rho}{c_s^2} & \frac{1}{c_s^2}\\
  0 & 0 & 0 & \frac{1}{c_s^2} & \frac{\tau_\Pi v_3}{\zeta c_s^2}
   \end{pmatrix}, \mathcal{B}= \begin{pmatrix}
 0  & 0 & 0 & 0 &0  \\
  0 & 0 & 0 & 0 & 0\\
  0 & 0 & 0 & 0 & 0\\
  0 & 0 & 0 & 0 & 0\\
  0 & 0 & 0 & 0 & \frac{1}{\zeta c_s^2}.
   \end{pmatrix}.
\eea
Note that we have assumed that $\zeta$, $c_s$, and $\rho$ are nonzero.
Given that the matrices are real and symmetric, to show that these equations are FOSH, we need to find the conditions under which $\mathcal{A}^0$ is invertible. This matrix is diagonal and invertible when $\tau_\Pi$, $c_s$, and $\zeta$ are nonzero. In particular, based on our stability analysis in the previous section, it is clear that if one assumes $\tau_\Pi>0$, $\zeta>0$, and $\rho>0$, the equations of motion are FOSH and also locally stable. 

In the Physics literature, it is customary to define a set of quasilinear PDEs such as the one in \eqref{definePDE} as \emph{strongly hyperbolic} when the inverse $(\mathcal{A}^0)^{-1}$ exists and the matrices $(\mathcal{A}^0)^{-1}\mathcal{A}^i$ are diagonalizable with a set of real eigenvalues and a corresponding set of linearly independent right eigenvectors \cite{Rezzolla_Zanotti_book}. According to this definition, every FOSH system is strongly hyperbolic, as is our system of equations. Furthermore, strongly hyperbolic PDEs are known to have a locally well-posed initial-value problem, which means that solutions of our equations of motion exist (at least for some time) and are unique for suitably defined initial data \cite{Taylor3,ChoquetBruhatGRBook,Bemfica-Disconzi-Graber-2021}, although local well-posedness for FOSH can be proven without appealing to their strong hyperbolic character \cite{KatoQuasiLinear}.

Given that the system is FOSH, its characteristic velocities should be real and finite. We shall now compute the characteristics for the FOSH system given by \eqref{nonrel_eq_mot1}, \eqref{nonrel_eq_mot2}, and \eqref{nonrel_eq_mot3}. By making the standard identifications $\partial_t \to \xi_0$ and $\partial_i \to \xi_i$ \cite{ChoquetBruhatGRBook}, the matrix defining the principal part of \eqref{definePDE} is given by 
\begin{equation}
L = \xi_0 \mathcal{A}^0 + \xi_i \mathcal{A}^i  = \begin{pmatrix}
\frac{1}{\rho} \alpha & \xi_1 & \xi_2 &\xi_3 & 0\\
\xi_1 & \frac{\rho}{c_s^2} \alpha & 0 & 0 & \frac{\xi_1}{c_s^2}\\
\xi_2 & 0 & \frac{\rho}{c_s^2} \alpha & 0 & \frac{\xi_2}{c_s^2}\\
\xi_3 & 0 & 0 & \frac{\rho}{c_s^2} \alpha &\frac{\xi_3}{c_s^2}\\
0 &\frac{\xi_1}{c_s^2} &\frac{\xi_2}{c_s^2} &\frac{\xi_3}{c_s^2} &\frac{\tau_\Pi \alpha}{\zeta c_s^2}\end{pmatrix},
\end{equation}
where $\alpha = \xi_0 + \xi_i v^i$. The determinant of this matrix is given by
\begin{equation}
    \det L = \frac{\alpha^3 \tau_\Pi}{\zeta c_s^8}\left(\alpha^2 - c_v^2\,\xi_i \xi^i \right)
\end{equation}
where 
\be
c_v =\sqrt{c_s^2 + \frac{\zeta}{\rho \tau_\Pi}}.
\ee
Solving $\det L =0$ for $\alpha$, we obtain that 
either
\begin{equation}
\xi_0 = -v_i \xi^i
\end{equation}
or
\begin{equation}
\xi_0 = -v_i \xi^i \pm c_v \sqrt{\xi_i \xi^i}.
\end{equation}
These are real solutions if we assume that $\zeta, \rho, \tau_\Pi > 0$, i.e., the standard physical assumptions. Local well-posedness in $H^{\frac{5}{2}^+}$ of this FOSH set of PDEs follow from well-known results \cite{KatoQuasiLinear}. 
To accommodate conditions at infinity suitable for fluid applications (e.g., density approaching a non-zero constant at infinity), it is more convenient to work with uniformly local Sobolev spaces
$H^N_{\text{u.l.}}$, for which the results  \cite{KatoQuasiLinear} also apply. More precisely, from the foregoing, we have:

\begin{theorem} \label{T:LWP_bulk}
Let $\rho^0(\mathbf{x})=\rho(0,\mathbf{x})$, $\mathbf{v}^0(\mathbf{x})=\mathbf{v}(0,\mathbf{x})$,  
$\Pi^0(\mathbf{x})=\Pi(0,\mathbf{x})$ be initial data for the system  
\eqref{nonrel_eq_mot1}, \eqref{nonrel_eq_mot2}, and \eqref{nonrel_eq_mot3}. 
Suppose that $(\rho^0,\mathbf{v}^0,\Pi^0)(\mathbb{R}^3)$ is contained in the interior of a compact set\footnote{In other words, this means that along $\{t=0\}$ the data functions take values within a compact subset of state space.}
$K \subset \mathbb{R}\times\mathbb{R}^3\times\mathbb{R}$. Assume that $P = P(\rho), \tau_\Pi = \tau_\Pi(\rho,\Pi), \zeta = \zeta(\rho,\Pi)$ are smooth functions of their arguments and that there exists a constant $C>1$ such that $\rho^0 \geq 1/C$ and  $1/C \leq c_s(\rho^0) \leq C$, $1/C \leq \tau_\Pi(\rho^0,\Pi^0) \leq C$,
$1/C \leq \zeta(\rho^0,\Pi^0) \leq C$.
Finally, suppose that $(\rho^0,\mathbf{v}^0,\Pi^0) \in H^{\frac{5}{2}^+}_{\text{u.l.}}(\mathbb{R}^3)$.
Then, there exists a $T>0$ and a unique solution 
\begin{align}
(\rho, \mathbf{v}, \Pi) \in C^0([0,T], H^{\frac{5}{2}^+}_{\text{u.l.}}(\mathbb{R}^3) )\cap
C^1([0,T], H^{\frac{3}{2}^+}_{\text{u.l.}}(\mathbb{R}^3) )
\nonumber
\end{align}
to equations \eqref{nonrel_eq_mot1}, \eqref{nonrel_eq_mot2}, and \eqref{nonrel_eq_mot3} taking the given
initial data $(\rho^0,\mathbf{v}^0,\Pi^0)$. Moreover, the solution depends on continuously 
on the initial data relative to the $C^0([0,T], H^{\frac{5}{2}^+}_{\text{u.l.}}(\mathbb{R}^3) )$ topology. 
\end{theorem}

We observe that, due to the Sobolev embedding theorem, solutions given by Theorem \ref{T:LWP_bulk} are  $C^1$ (i.e., they are continuously differentiable and thus classical solutions). We also note that the condition for hyperbolicity found in the non-relativistic model, $\zeta, \rho, \tau_\Pi,c_s > 0$, is simpler than the condition for hyperbolicity found in the full relativistic case in \eqref{eqPRLbulk}. Also, note the non-relativistic regime's lack of an upper bound on $\zeta$ and $\tau_\Pi$. This occurs because, differently than in the relativistic case, we do not have to impose that the characteristic velocities are subliminal. Finally, we note that $\tau_\Pi$ and $\zeta$ can be general functions not only on the density but also on the bulk stress $\Pi$. Thus, our results are valid for a vast class of non-Newtonian fluids parameterized by the dependence of the transport coefficients on the viscous stress.

%%%%%%%%%%%%%%%%%%%%%%%%%%%%%%%%%%%%%%%%%%%%%

\section{Breakdown of classical solutions\label{GWPSec}}

%%%%%%%%%%%%%%%%%%%%%%%%%%%%%%%%%%%%%%%%%%%%%

In this section, we investigate the global properties of the solutions of our system defined by \eqref{nonrel_eq_mot1}, \eqref{nonrel_eq_mot2}, and \eqref{nonrel_eq_mot3}. We showed in the previous section that LWP holds when $\zeta, \rho, \tau_\Pi,c_s > 0$. We now show that there is initial data for which $C^1$ solutions to the equations in three dimensions do not exist for all time. This is the first result of this kind for the broad class of non-Newtonian fluids considered in this paper. This is done by adapting the well-known result of Sideris proved in Ref.\ \cite{1985CMaPh.101..475S} to our system of equations describing a non-relativistic compressible bulk-viscous fluid. 

Consider the system described by Eqs.\ \eqref{nonrel_eq_mot1}, \eqref{nonrel_eq_mot2}, and \eqref{nonrel_eq_mot3}. Our equations can be considered a hyperbolic generalization of Euler equations with bulk viscosity. Thus, we will extend the proof of Theorem 1 of \cite{1985CMaPh.101..475S} to include viscous effects. The basic idea behind this theorem is to show that a suitably defined quantity related to the fluid's total momentum obeys a differential inequality whose solution possesses a finite life span.

We now introduce the following assumptions: assume that all the fields $\{\rho,v,\Pi\}$ depend on $t$ and $\mathbf{x}$, and that $\rho,\zeta,c_s, \tau_\Pi>0$. We further assume that the equation of state is $P(\rho) = A\, \rho^\gamma$, with constant $\gamma >1$ and $A>0$, and that $\zeta$ and $\tau_\Pi$ are positive constants. The latter assumption is prevalent in the literature \cite{LandauLifshitzFluids}, implying that the transport coefficients do not vary appreciably for the densities considered. 

We start by defining
\begin{equation}
\label{F_def}
    F(t) := \int_{\R^3} d^3 \mathbf{x}\,x_i\, \rho\,   v^i.
\end{equation}
This, and other integral quantities defined below, are easily seen to be finite under the assumptions of our theorems.
The quantity above corresponds to an average of the radial component of the momentum of the fluid, which remains finite as long as the solution is $C^1$. We show below that there is initial data for which $F$ is not bounded for an arbitrarily large time. 

We define our initial data as follows. Let $R > 0$ be the radius of some ball in $\mathbb{R}^3$ containing all the fluid at the initial time $t=0$. The initial data is given by $\rho(t = 0,\mathbf{x}) = \rho^0(\mathbf{x}) > 0$, $\mathbf{v}(t = 0,\mathbf{x}) = \mathbf{v}^0(\mathbf{x})$,  and  $\Pi(t = 0,\mathbf{x}) = \Pi^0(\mathbf{x})$, assumed to be $H^{\frac{5}{2}^+}$ (thus in particular $C^1$) functions in $\mathbb{R}^3$. Following \cite{1985CMaPh.101..475S}, we consider constant initial data outside the bounded set such that $\rho^0 = \bar\rho >0$, $\mathbf{v}^0(\mathbf{x}) =\bar{\mathbf{v}}$, and  $\Pi^0(\mathbf{x}) =\bar\Pi$ for $|\mathbf{x}| \geq R$. 

The results of the previous section guarantee that for this initial data, LWP holds for sufficiently regular data. However, for the following proof to hold, it is sufficient to assume that we have a $C^1$ solution, which will be the case, for example, for initial data in $H^{\frac{5}{2}^+}$, as already noted. Since $\bar{\mathbf{v}}=0$, the maximum propagation of the front of a smooth disturbance is governed by 
\be
\bar{c}_v = \sqrt{\bar{c}_s^{\,2} + \frac{\bar\zeta}{\bar\rho \,\bar{\tau}_\Pi}}
\ee
where 
\be
\bar{c}_s = \sqrt{A \gamma \bar\rho^{\,\gamma-1}}
\ee
is the (ideal fluid) speed of sound at the front, and $\bar\zeta$ and $\bar\tau_\Pi$ are determined by $\bar\rho$ and $\bar\Pi$. Thus, in terms of the domain \cite{1985CMaPh.101..475S}
\be
B(t) = \{\mathbf{x}\in \mathbb{R}^3: |\mathbf{x}|\leq R+ \bar{c}_v \,t \},
\ee
which is the region containing the non-constant part of the fluid, 
the hyperbolic nature of our equations implies that  $(\rho, \mathbf{v},\Pi) = (\bar\rho,0,\bar\Pi)$ outside $B(t)$. Furthermore, one can use \eqref{nonrel_eq_mot1} to show that the relative mass in $B(t)$ with respect to the background
\be
\label{DeltaM_def}
\Delta M(t) = \int_{\mathbb{R}^3}d^3\mathbf{x}\,\left(\rho -\bar\rho\right) =  \int_{B(t)}d^3\mathbf{x}\,\left(\rho -\bar\rho\right),
\ee
is constant in time. Finally, we define the average contribution from the non-equilibrium part of the pressure
\be
\label{G_def}
G(t) = \int_{\R^3}d^3\mathbf{x}\,\Pi.
\ee

We can now state the following result: 
%Theorem: 
\begin{theorem}\label{thm1} Suppose that $(\rho, \mathbf{v},\Pi)$ is a $C^1$  solution of \eqref{nonrel_eq_mot1}, \eqref{nonrel_eq_mot2}, and \eqref{nonrel_eq_mot3}. Suppose that
 $\rho(0,\mathbf{x}) = \rho^0(\mathbf{x}) > 0$, and that for some $R>0$,  
 $\rho^0(\mathbf{x}) = \bar\rho >0$, $\mathbf{v}(0,\mathbf{x}) = \mathbf{v}^0(\mathbf{x}) = 0$,  $\Pi(0,\mathbf{x})=\Pi^0(\mathbf{x}) = 0$ for $|\mathbf{x}| \geq R$, where $\bar{\rho}$ is a constant.
Assume that $\Delta M(0) \geq 0$, $F(0) > \frac{16\pi}{3} \bar{c}_v R^4 \max \rho^0(x)$, and $G(0)\geq 0$, where $\Delta M$, $F$, and $G$ are given by \eqref{DeltaM_def}, \eqref{F_def}, and \eqref{G_def}, respectively. 
Assume that the equation of state\footnote{The key property of the equation of state that is employed in the proof is convexity. More general, convex equations of state can be used with little or no change to the proof, but we will not do it here for the sake of concreteness.} is given by $P(\rho) = A \rho^\gamma$, where $\gamma>1$ and $A>0$ are constants.
Finally, suppose that $\tau_\Pi$ and $\zeta$ are positive constants.
Then, the lifespan\footnote{We define the lifespan of $C^1$ solutions in the usual way, i.e., as the supremum over all $\mathcal{T}>0$ such that the system admits a $C^1$ solution defined on $[0,\mathcal{T}]$.} of $(\rho, \mathbf{v},\Pi)$ is finite.
\end{theorem}

\begin{proof} 
We prove the statement using proof by contradiction. We will show that  for some
$\mathcal{T}>0$ the quantity $F'(\mathcal{T})$ does not exist.

Suppose that $\rho, \mathbf{v}, \Pi$ are $C^1$ functions of $\mathbf{x}$ and $t$. Recall that $\bar{\mathbf{v}}=0$ and $\bar{\Pi}=0$. Using integration by parts, we obtain 
\begin{equation}
    \partial_t F(t) = \int_{B(t)} d^3\mathbf{x}\, \left[\rho |\mathbf{v}|^2 + 3(P - \bar{P})\right] + 3 G(t),
\end{equation}
where $\bar{P} = P(\bar{\rho})$. 
We will show below
 that $\int_{B(t)} d^3\mathbf{x}\, \left(P -\bar{P}\right) + G(t)  > 0$ when $t \geq 0$, which implies 
that $\partial_t F(t) \geq \int_{B(t)} d^3\mathbf{x}\, \rho |\mathbf{v}|^2 >  0$ for all time.

Following \cite{1985CMaPh.101..475S}, since $\gamma>1$, one finds that
\begin{align}
    \int_{B(t)}d^3\mathbf{x}\, P 
    &= A\int_{B(t)}  d^3\mathbf{x}\,\rho^\gamma  \nonumber \\
    &\geq A (\mbox{vol }B(t))^{1 - \gamma}\left(\int_{B(t)} d^3\mathbf{x}\,\rho\right)^{\gamma} \quad \mbox{by Jensen's inequality} \nonumber \\
    &= A (\mbox{vol }B(t))^{1 - \gamma} \left(\Delta M + \bar\rho\, \mbox{vol }B(t)\right)^{\gamma}\nonumber\\
    &\geq A \,\bar\rho^{\textcolor{blue}{\gamma}} \,\mbox{vol } B(t)\nonumber\\
    &= \int_{B(t)} d^3\mathbf{x}\,\bar{P}  ,
    \label{Pproof}
\end{align}
where $\mbox{vol } B(t) = \frac{4\pi}{3} (R+\bar{c}_v t)^3$.
To obtain the desired estimates for $\partial_t F(t)$, we must investigate the sign of $G(t)$. 

Integrating the equation of motion for $\Pi$ in \eqref{nonrel_eq_mot3} to find
\be
\tau_\Pi \partial_t G(t)+ G(t) = - \int_{B(t)}d^3\mathbf{x}\,\zeta \,\partial_i v^i. 
\ee
Solving this ODE for $G$, we have
\be
G(t)  = e^{-\frac{t}{\tau_\Pi}} G(0) - \frac{1}{\tau_\Pi}\int^t_0 dt'\,e^{-(t-t')/\tau_\Pi}\int_{B(t')}d^3\mathbf{x}\,\zeta\, \partial_i v^i(t^\prime,\mathbf{x}). 
\ee
This expression is general. But under our assumption that $\zeta$ is constant, 
we can integrate $\partial_i v^i$ by parts, and the last term on the right-hand side vanishes. Thus, when $G(0)\geq 0$, we have $G(t)\geq 0$. Therefore, we have shown that 
\be
\partial_t F(t) \geq \int_{B(t)} d^3\mathbf{x}\, \rho |\mathbf{v}|^2 > 0. 
\ee
Following \cite{1985CMaPh.101..475S}, we use the Cauchy-Schwarz inequality to find
\begin{equation}
    F(t)^2 \leq \left(\int_{B(t)}d^3\mathbf{x}\, |\mathbf{x}|^2 \rho \right)\left( \int_{B(t)} d^3\mathbf{x}\,\rho |\mathbf{v}|^2\right),
\end{equation}
and in addition, we have
\begin{equation}
    \int_{B(t)}d^3\mathbf{x}\, |\mathbf{x}|^2 \rho  \leq \frac{4\pi}{3} (R + \bar{c}_v t)^5 \max \rho^0(x),
\end{equation}
hence
\begin{equation}
    \partial_t F(t) \geq \left(\frac{4\pi}{3} (R + \bar{c}_v t)^5 \max \rho^0\right)^{-1} F(t)^2.
\end{equation}
However, dividing by $F(t)^2$ and integrating from $0$ to $\mathcal{T}$, we obtain
\begin{equation}
    F(0)^{-1} \geq F(0)^{-1} - F(\mathcal{T})^{-1} \geq \left(\frac{16\pi}{3} \bar{c}_v\max \rho^0\right)^{-1} (R^{-4} - \left(R + \bar{c}_v \mathcal{T})^{-4}\right)
\end{equation}
which contradicts the original assumption that $F(0)\geq \frac{16\pi}{3} \bar{c}_v R^4 \max \rho^0(x)$
if $\mathcal{T}$ is very large. Hence, the time of existence of the solution must be finite.  
\end{proof}

%%%%%%%%%%%%%%%%%%%%%%%%%%%%%%%%%%%%%%%
\section{Full Case with Shear and  bulk Viscosities}\label{GWPshear}
%%%%%%%%%%%%%%%%%%%%%%%%%%%%%%%%%%%%%%%

The effects of shear viscosity are significant in studying relativistic fluid dynamics. Many systems, such as gases and incompressible fluids, exhibit shear viscous effects. 
This section considers the full nonrelativistic equations, including the effects of shear viscosity and bulk viscosity. After casting the equations in a strongly hyperbolic form, our results in this section closely
parallel those of the previous sections. Thus, our presentation will be brief.

%%%%%%%%%%%%%%%%%%%%%%%%%%%%
\subsection{Equations of motion and local well-posedness}

Here, we briefly describe the equations of motion in the shear case and local well-posedness results. As in Section \ref{section2}, one starts by considering the relativistic Israel-Stewart equations and taking the non-relativistic limit of these equations. We assume, for simplicity, that the shear relaxation time is equal to $\tau_\Pi$ and that the second-order transport coefficients are such that  the equations of motion for the fluid are in the end given by: 
\begin{align}
   \frac{1}{\rho} \partial_t \rho + \partial_k v^k + \frac{1}{\rho} v^k \partial_k \rho &= 0 \label{sheareq1}\\
   \frac{\rho}{c_s^2} \partial_t v_i + \partial_i \rho + \frac{\rho}{c_s^2} (v \cdot \partial) v_i + \frac{1}{c_s^2} \partial_j \Pi_{i}^j &= 0 \label{sheareq2}\\
   \tau_\Pi\partial_t \Pi_{ij} + \eta \partial_i v_j +\eta \partial_j v_i+ \delta_{ij}\left(\zeta-\frac{2\eta}{3}\right)\partial_k v^k
   + \tau_\Pi \partial_k(v^k \Pi_{ij}) + \Pi_{ij} &= 0. \label{sheareq4}
\end{align}
Here $\eta$ is the shear viscosity transport coefficient. This describes a class of non-Newtonian fluids with shear and bulk viscosity, where the viscous stress tensor relaxes towards its Navier-Stokes value. In this context, our equations resemble Grad's 13-moments theory \cite{Grad1958}, though we remark that in our case, there is a nonzero bulk viscosity (given that $\Pi_{ij}$ is not traceless) but no heat flux. Furthermore, we note that in our case, $\tau_\Pi$, $\zeta$, and $\eta$ can depend on $\rho$ and, also, on $\Pi_{ij}$, which contributes to the generality of our results. In addition, we note that the normalized trace of the viscous stress tensor, $\Pi \equiv \Pi_i^i/3$, obeys Eq.\ \eqref{nonrel_eq_mot3}. Finally, we stress that these are first-order nonlinear equations with a very similar structure to those discussed in Section \ref{section2}. While they no longer form a symmetric hyperbolic system (unless $\zeta=2\eta/3$), we will find that they have similar properties to the purely bulk-viscous case.

%\subsection{Local Well-Posedness and Hyperbolicity}
We now show that this non-Newtonian fluid with bulk and shear viscosity is strongly hyperbolic and, hence, locally well-posed for initial data in a sufficient Sobolev space. After rewriting the system in quasilinear form
\be
\mathcal{A}^0 \partial_t \Phi + \mathcal{A}^1 \partial_1 \Phi + \mathcal{A}^2 \partial_2 \Phi+\mathcal{A}^3 \partial_3 \Phi + \mathcal{B} \Phi = 0,
\ee
we may compute the characteristic matrix. Direct computation shows that it is diagonalizable, with all eigenvalues being real. The characteristic polynomial of the system is
\begin{equation}
    \frac{ \tau_\Pi^3 \alpha^4}{c_s^6 \rho} \left(\alpha^2 \rho \tau_\Pi - \eta \xi_i \xi^i\right)\left(\alpha^2 \rho \tau_\Pi - (\zeta + \frac{4}{3} \eta + c_s^2 \rho \tau_\Pi) 
 \xi_i \xi^i\right),
\end{equation}
where $\alpha = \xi_0 + v_i  \xi^i$, as in Section \ref{section2}.
We solve for $\alpha$, finding
\begin{equation}
     \alpha = 0, \sqrt{\frac{\eta}{\rho \tau_\Pi}}, \sqrt{c_s^2 + \frac{\zeta + \frac{4}{3}\eta}{\rho \tau_\Pi}}
\end{equation}
as our solutions. Since we have a first-order hyperbolic equation with a diagonalizable matrix, we conclude that it is 
locally well-posed in $H^{\frac{5}{2}^+}$ and in $H_{\text{u.l.}}^{\frac{5}{2}^+}$ \cite{Taylor3,ChoquetBruhatGRBook,Bemfica-Disconzi-Graber-2021}. In particular, the corresponding solutions are $C^1$. For the sake of brevity, we will not state a detailed LWP theorem, but the interested reader should have no difficulty mimicking Theorem \ref{T:LWP_bulk} for this case. In contrast to the full viscous case in the relativistic regime, where local well-posedness is not yet proven, we proved local well-posedness under basic physical assumptions in the non-relativistic limit. Once again, we did not need to assume any upper bounds on $\zeta, \Pi_{ij}, \tau_\Pi, \eta, \rho$. Furthermore, local well-posedness follows as long as the typical physical conditions are employed, i.e., $\rho, \tau_\Pi, \zeta, \eta > 0$. In particular, we also remark that hyperbolicity follows even if $\tau_\Pi$ and $\zeta$ and $\eta$ are nontrivial functions of the density and the stresses. We note that isotropy dictates that a dependence of the transport coefficients on $\Pi_{ij}$ can only appear via the rotationally invariant combinations $\Pi_{ij}\Pi^{ij}$ and  $\Pi = \Pi^{i}_i/3$. Thus, we see that our local well-posedness results for the system with shear and bulk viscosity are valid for a very general class of non-Newtonian fluids parameterized by the dependence of the transport coefficients on the stress tensor.

One can similarly establish local stability results, as in Section \ref{bulklocal}. For completeness, we show below in Section~\ref{appendixA} that if $\rho, \tau_\Pi, \zeta, \eta > 0$, the system is stable under perturbations around the equilibrium state. 

%%%%%%%%%%%%%%%%%%%%%%%%%%%%%%%%%%%%%%%%%%%%%%%%%%%%%%%%%%%%%%%

\subsection{Local stability around the equilibrium state}\label{appendixA}
In this section, we prove the local stability of \eqref{sheareq1}-\eqref{sheareq4} under small perturbations from equilibrium. As in section \ref{bulklocal}, we begin by taking the Fourier transform of our equations and linearizing around the equilibrium where $\Pi_{ij}$ vanishes. This procedure grants us the following system of equations:
\begin{align}
    -i ( \omega - \mathbf{v_0} \cdot \mathbf k)  \delta \rho + i \rho_0 \mathbf k \cdot \delta \mathbf v  &= 0 \\
     - i \frac{\rho_0}{c_s^2} (\omega - \mathbf{v_0} \cdot \mathbf k)\delta v_i + i k_j  \delta \rho  + \frac{1}{c_s^2} k_j \delta \Pi_{ij} &= 0 \\
    -i(\omega - \mathbf{v_0} \cdot \mathbf k) \tau_\Pi \delta \Pi_{jm} +   \delta \Pi_{jm} + i\eta k_j \delta v_m + i\eta k_m \delta v_j + i \delta_{jm}\left(\zeta - \frac{2\eta}{3}\right)k_\ell \delta v^\ell &= 0.
\end{align}
We now re-express this as a matrix equation
\begin{equation}
    C \Psi = 0
\end{equation}
where $\Psi = \begin{pmatrix}
  \delta \rho \\ \delta v_i \\ \delta \Pi_{ij}
\end{pmatrix}$
is a 10-vector. This equation has solutions if and only if $\det C = 0$. This gives us a polynomial, which we may use to solve for the value of $\Omega= \omega - \mathbf{v_0} \cdot \mathbf k$. This gives us the equation
\begin{equation}
    \frac{(1 - \tau_\Pi i \Omega)^3}{3 c_s^6 \rho}\left(-i\Omega (1 - i \tau_\Pi \Omega) \rho  + \eta \mathbf{k}^2 \right)\left(\rho \tau_\Pi(-i\Omega)^3  + \rho (-i\Omega)^2 + (-i\Omega)(3\zeta + 4\eta + c_s^2 \rho \tau_\Pi)\mathbf{k}^2 + c_s^2 \rho \mathbf{k}^2 \right) = 0.
\end{equation}
We now determine the conditions for $-i\Omega$ to have a positive real part, as we did in Section \ref{bulklocal}.
The linear term grants us that $\tau_\Pi > 0$, and using that with the quadratic term, we must also have $\rho > 0$, $\eta > 0$. Finally, using the Routh-Hurwitz condition for the cubic term, as we did in section \ref{bulklocal}, we obtain that $-i\Omega$ is negative if and only if $\tau_\Pi$, $\rho$, $\zeta$, $\eta > 0$, as desired. Therefore, we have shown that the full viscous non-Newtonian fluid is locally stable under the basic physical assumptions of $\tau_\Pi$, $\rho$, $\zeta$, $\eta > 0$.

\subsection{Breakdown of classical solutions}
This section extends Theorem \ref{thm1} to the full viscous case. 
The setting is very similar to that of Theorem \ref{thm1}, the main difference being that 
the assumption $\int_{\R^3} \Pi^0 \, d^3 \mathbf{x}\geq 0$ is now replaced by $\int_{\R^3} (\Pi^0)^i_i , d^3 \mathbf{x} \geq 0$. We do not need further assumptions on the other viscous shear stress tensor components or $\eta$. Therefore, the results in this section are very general when it comes to the properties of shear-viscous stresses because the shear viscosity coefficient $\eta$ can be a very general  function\footnote{It is physical to assume that $\eta>0$ for stability.} of $\Pi$ and $\Pi_{ij}\Pi^{ij}$, parameterizing a large class of non-Newtonian fluids.

\begin{theorem}
Suppose that $(\rho, \mathbf{v},\Pi_{ij})$ is a $C^1$ solution of \eqref{sheareq1}-\eqref{sheareq4}.
Suppose that
 $\rho(0,\mathbf{x}) = \rho^0(\mathbf{x}) > 0$, and that for some $R>0$,  
 $\rho^0(\mathbf{x}) = \bar\rho >0$, $\mathbf{v}(0,\mathbf{x}) = \mathbf{v}^0(\mathbf{x}) = 0$,  $\Pi_{ij}(0,\mathbf{x})=\Pi^0_{ij}(\mathbf{x}) = 0$ for $|\mathbf{x}| \geq R$, where $\bar{\rho}$ is a constant.    
Assume that $\Delta M(0) \geq 0$, $F(0) > \frac{16\pi}{3} \bar{c}_v R^4 \max \rho^0(x)$, and $G(0)\geq 0$, where $\Delta M$ and $F$ are given by \eqref{DeltaM_def} and \eqref{F_def}, respectively, and 
\begin{align}
\nonumber
G(t) := \int_{\mathbb{R}^3} \Pi^i_i \, d^3\mathbf{x}.
\end{align}
Assume that the equation of state is given by $P(\rho) = A \rho^\gamma$, where $\gamma>1$ and $A>0$ are constants.
Finally, suppose that $\tau_\Pi$ and $\zeta$ are positive constants.
Then, the lifespan of  $(\rho, \mathbf{v},\Pi_{ij})$ is finite.
\end{theorem}

\begin{proof}
Direct computation yields
\begin{align}
    \partial_t F(t) &=  \int_{B(t)} \rho |\mathbf{v}|^2 \, d^3\mathbf{x} + \int_{B(t)} 3(P - \bar P)\, d^3\mathbf{x} + \int_{B(t)} \Pi^i_i \, d^3\mathbf{x}
\end{align}
As the equation for $\Pi^i_i$ is the same as \eqref{nonrel_eq_mot3} from the purely bulk-viscous case, we may apply the rest of the proof and obtain the desired result.
\end{proof}

We emphasize that this is the first result of this kind in the literature. The theorem above gives precise conditions under which the solutions of the equations of motion of the system with shear and bulk viscosity lose $C^1$ status. 

%%%%%%%%%%%%%%%%%%%%%%%%%%%%%%%%%%%%%%%
\section{Conclusions and outlook}\label{conclusions}
%%%%%%%%%%%%%%%%%%%%%%%%%%%%%%%%%%%%%%%

In this paper, we considered a large class of non-Newtonian compressible fluids with shear and bulk viscosities, where the viscous stress tensor relaxes towards its Navier-Stokes form. After pointing out that such fluids can formally be derived from the non-relativistic limit of well-known Israel-Stewart-like theories, we considered their evolution problem. We showed that the Cauchy problem is locally well-posed in standard functions spaces but that it is not generally globally well-posed, as we have constructed initial data for which the solution's lifespan is finite. Interestingly, the shear part of the viscous stress tensor did not actively participate in the global well-posed analysis because the presence of singularities at finite time followed directly from the properties of the bulk viscosity sector. In fact, one could use the results presented here to investigate global well-posedness in an even larger class of theories, where the bulk sector would lead to singularities and the shear channel would act as a ``spectator" in the analysis. Therefore, our results show for the first time in the literature that a large class of non-Newtonian fluids, parameterized by the dependence of the transport coefficients on the stress tensor, have solutions that lose regularity at finite time. This is the first result of this kind when it comes to the global well-posedness properties of solutions of such systems. Since they correspond to the non-relativistic regime of the Israel-Stewart theories currently used in the description of the quark-gluon plasma formed in heavy-ion collisions, our work shows, for the first time, that the regions in the quark-gluon plasma where the flow velocity approaches the non-relativistic limit can lose regularity at finite time. Additionally, a similar statement can be made about the slow (i.e., non-relativistic) moving parts of the matter formed in neutron star mergers that have been described in terms of Israel-Stewart-like theories \cite{Camelio:2022ljs,Camelio:2022fds,Most:2022yhe,Chabanov:2023blf}. Therefore, our work reveals new mathematical properties of a large class of non-Newtonian fluids and leads to a deeper understanding of the non-relativistic regime of relevant systems in high-energy nuclear physics and astrophysics. 

A better understanding of the mathematical properties of this theory can be beneficial in analyzing the mathematical behavior of second-order theories of the Israel-Stewart type. While both theories share many features, such as the use of extended variables and a description of the evolution of the viscous component of the fluid via relaxation-type equations, the non-relativistic, non-Newtonian equations considered here are significantly simpler. Thus, understanding its mathematical features can be a valuable guide for the mathematical study of its relativistic counterpart. This is especially the case given that many important mathematical questions, such as LWP, remain open for the latter under general conditions, as mentioned in the introduction. 

In addition, it is interesting to notice that the class of non-Newtonian fluids studied here resembles (in essence) Grad's 13-moment theory and other extended variables approaches \cite{MuellerRuggeriBook,JouetallBook}. The latter are well-studied models of non-Newtonian fluids, which deserve full mathematical investigation on their own. In that context, it is known that Grad's equations lose hyperbolicity for significant deviations from equilibrium \cite{MuellerRuggeriBook}. The theories considered here remain strongly hyperbolic, even arbitrarily far from equilibrium. However, as mentioned above, the solutions can develop singularities at a finite time.   

    Finally, as mentioned previously, our results have direct consequences for the non-relativistic regime of non-Newtonian fluids found in high-energy nuclear physics and astrophysics. It would also be interesting to investigate applications of our results to other non-Newtonian fluids \cite{Abu-Nab_2021,Abu-Nab_2022}. Furthermore,  one may also explore how turbulence emerges in the class of non-Newtonian fluids considered here. Many studies of turbulence in non-Newtonian fluids exist in the literature, for instance, Ref.\ \cite{MORAD2021101527}. Differently than the standard case studied within the Navier-Stokes equations, where the transition from the viscous to the Eulerian limit corresponds to going from parabolic to hyperbolic equations, here both the ideal and viscous regimes are described by a hyperbolic set of equations of motion. Thus, it is conceivable that studying the transition to turbulence in the model considered here may shed light on the much less understood problem of relativistic turbulence \cite{Eyink:2017zfz,Calzetta:2020wzr}, where coupling to Einstein's equations requires that the fluid equations of motion must possess a locally well-posed initial value problem \cite{ChoquetBruhatGRBook}, regardless of whether there is dissipation or not. In particular, one may study fully developed turbulence in the non-Newtonian models considered here following Ref.\ \cite{Calzetta:2020wzr}. We leave these questions to future work.

\section*{Acknowledgments} The authors thank L.~Gavassino for discussions. AL thanks the Barry M. Goldwater Scholarship for support.  
MMD is partially supported by NSF grant DMS-2107701, a Chancelor's Faculty Fellowship, and a Vanderbilt Seeding Success Grant.
JN is partially supported by the U.S. Department of Energy, Office of Science, Office for Nuclear Physics under Award No. DE-SC0021301 and DE-SC0023861.  MMD and JN thank the KITP Santa Barbara for its hospitality during ``The Many Faces of Relativistic Fluid Dynamics" Program, where this work's last stages were completed. This research was partly supported by the National Science Foundation under Grant No. NSF PHY-1748958.

\bibliography{References.bib}
\end{document}